\documentclass{amsart}
\usepackage{shortsalch, xy}
\xyoption{all}
\title{Graded comodule categories with enough projectives.}
\author{A. Salch}
\begin{document}

\begin{abstract}
It is well-known that the category of comodules over a flat Hopf algebroid is abelian but typically fails to have enough projectives, and more generally, the category of graded comodules over a graded flat Hopf algebroid is abelian but typically fails to have enough projectives. 
In this short paper we prove that the category of connective graded comodules over a connective, graded, flat, finite-type Hopf algebroid has enough projectives. Applications to algebraic topology are given: the Hopf algebroids of stable co-operations in complex bordism, Brown-Peterson homology, and classical mod $p$ homology all have the property that their categories of connective graded comodules have enough projectives. We also prove that categories of connective graded comodules over appropriate Hopf algebras fail to be equivalent to categories of graded connective modules over a ring.
\end{abstract}

\maketitle
\tableofcontents

\section{Introduction.}

Let $(A,\Gamma)$ be a graded Hopf algebroid (that is, a cogroupoid object in the category of graded-commutative rings) such that $\Gamma$ is flat over $A$. 
Then the category of graded $\Gamma$-comodules is abelian, and homological algebra in this category is of central importance in algebraic topology, since the input for generalized Adams spectral sequences is a (relative) $\Ext$ functor in a category of graded $\Gamma$-comodules; see chapters~2 and~3 of~\cite{MR860042} for a textbook account of this material. Appendix~1 of~\cite{MR860042} is the standard reference for Hopf algebroids and homological algebra in their comodule categories.

Some homological constructions in comodule categories are made problematic, however, by the lack of enough projectives. It is well-known that the category of comodules over a Hopf algebroid typically fails to have enough projectives; even when $A$ is a field and $\Gamma$ a Hopf algebra over $A$, the category of $\Gamma$-comodules has enough projectives if and only if $\Gamma$ is {\em semiperfect}, i.e., every simple comodule has an injective hull which is finite-dimensional as an $A$-vector space. (This result is attributed by B.~I. Lin, in~\cite{MR0498663}, to unpublished work of Larson, Sweedler, and Sullivan; the generalization of this result which replaces Hopf algebras with coalgebras is a result of Lin's, from the same paper.)

Here is an example: in the paper~\cite{MR2066503} (see the Remark preceding Proposition~1.2.3), M. Hovey shows that the category of comodules over the Hopf algebra $\mathbb{Q}[x]$, with $x$ primitive, has the property that infinite products fail to be exact. That is, Grothendieck's axiom AB$4^*$ fails in this category of comodules. It is standard that a complete abelian category which has enough projectives also satisfies axiom AB$4^*$ (see e.g. Lemma~A.3.15 of~\cite{MR1812507}), so this category of comodules cannot have enough projectives. Hovey's example also works in the graded case (although, crucially, not the {\em connective} graded case, if $x$ is in a positive grading degree).

The purpose of this short paper is to prove that, under some reasonable assumptions (which are satisfied in cases of topological interest), appropriate categories of graded comodules over graded Hopf algebroids {\em do} have enough projectives.
The essential point is to work with {\em connective} graded comodules, that is, graded comodules which are trivial in all negative grading degrees; the category of connective graded comodules over the Hopf algebra $\mathbb{Q}[x]$ of Hovey's example {\em does} have enough projectives, and much more generally, our main result is Theorem~\ref{main thm}:
\begin{unnumberedtheorem}
Let $(A,\Gamma)$ be a connective finite-type flat graded Hopf algebroid.
Then the category of connective graded $\Gamma$-comodules
is a Grothendieck category with a projective generator. 
Consequently, 
the category of connective graded $\Gamma$-comodules
has enough projectives and enough injectives, and satisfies Grothendieck's axiom
AB$4^*$ (that is, infinite products exist and are exact).
\end{unnumberedtheorem}
However, if $A$ is not the zero ring, then this category of connective graded $\Gamma$-comodules fails to have a {\em compact} projective generator, so it is {\em not} equivalent to the category of (ungraded) modules over any ring.\footnote{As a peculiar but elementary special case, which must certainly already be well-known: if $A = \Gamma$ with trivial grading, the category of connective graded $A$-modules---which is, of course, isomorphic to a countable infinite product of copies of the category $\Mod(A)$---is not equivalent to the category of modules over a ring.} This is proven in Proposition \ref{gr-comod isnt a module category}. 

Of course it is then natural to ask whether the category of $\Gamma$-comodules might be equivalent to the category of {\em connective graded} modules over a {\em connective graded} ring. This takes a bit more work: in Theorem \ref{main nonequivalence thm} we show, under some reasonable hypotheses on $\Gamma$, that the category of graded connective $\Gamma$-comodules cannot be equivalent by a suspension-preserving equivalence to the category of connective graded modules over a ring.

An amusing consequence is Corollary \ref{steenrod nonequivalence cor}: the category of connective graded comodules over the mod $p$ dual Steenrod algebra is not equivalent, via a suspension-preserving functor, to the category of connective graded modules over a ring. Nevertheless, that comodule category {\em does} have enough projectives, by Theorem~\ref{main thm}. 

Some terminology used above may not be immediately familiar. The relevant definitions are as follows: 
\begin{itemize}
\item a graded Hopf algebroid $(A,\Gamma)$ is {\em flat} if $\Gamma$ is flat over $A$,
\item {\em connective} if $A$ and $\Gamma$ are concentrated in nonnegative grading degrees (i.e., $(A,\Gamma)$ is $\mathbb{N}$-graded, not just $\mathbb{Z}$-graded),
\item and {\em finite-type} if there exists an exact sequence of graded $A$-modules
\[ \coprod_{i\in\mathbb{Z}} \Sigma^{i} A^{\oplus b_i}  \rightarrow \coprod_{i\in\mathbb{Z}} \Sigma^{i} A^{\oplus a_i} \rightarrow \Gamma \rightarrow 0 \]  for some sequences of integers $(\dots ,a_{-1},a_0,a_1, a_2,\dots)$ and $(\dots ,b_{-1},b_0,b_1, b_2,\dots)$. (Of course, if $\Gamma$ is also connective, then $a_i$ and $b_i$ each must vanish for sufficiently small $i$.)
\end{itemize}
Following the usual convention in topology, we write $\Sigma$ for the suspension operator, i.e., $\Sigma A$ is $A$ with all grading degrees increased by one.

Special cases of Theorem~\ref{main thm} include some of the most important Hopf algebroids for topological applications, as we see in Corollary~\ref{main cor}:
\begin{unnumberedcorollary}
The categories of connective graded comodules over the Hopf algebroids\linebreak
$(MU_*,MU_*MU)$, $(BP_*,BP_*BP)$, and $((H\mathbb{F}_p)_*, (H\mathbb{F}_p)_*H\mathbb{F}_p)$ all have enough projectives.
\end{unnumberedcorollary}
These Hopf algebroids are very well-known in algebraic topology: 
$(MU_*,MU_*MU)$ is the Hopf algebroid of stable natural co-operations of the complex bordism functor $MU_*$, $(BP_*,BP_*BP)$ is the Hopf algebroid of stable natural co-operations of the $p$-local Brown-Peterson homology functor $BP_*$, and 
$((H\mathbb{F}_p)_*, (H\mathbb{F}_p)_*H\mathbb{F}_p)$ is the mod $p$ dual Steenrod algebra, i.e., the Hopf algebra of stable natural co-operations of the mod $p$ classical homology functor $(H\mathbb{F}_p)_*$. These are the Hopf algebroids whose comodule categories have the most important homological invariants: appropriate relative $\Ext$ groups over these three Hopf algebroids recover the $E_2$-terms of the global Adams-Novikov, $p$-local Adams-Novikov, and classical $p$-primary Adams spectral sequences, respectively.
See chapters 2, 3, and 4 of~\cite{MR860042} for this material.

I am grateful to G. Valenzuela for useful conversations relating to this material, and to A. Baker and an anonymous referee for their patience with how long I took to make revisions on this paper.

\section{When does tensor product of modules commute with infinite products?}

\begin{convention}
In this paper, all gradings will be assumed to be $\mathbb{Z}$-gradings. When a graded object is trivial in all negative grading degrees, we will say that the object is {\em connective}. We write $\mathbb{N}$ for the set of nonnegative integers.
\end{convention}

\begin{definition}\label{def of finite-type}
Let $A$ be a graded ring. 
\begin{itemize}
\item
We will say that a graded $A$-module $M$ is {\em finite-type and free} if $M$ is a free $A$-module with finitely many generators in each degree. That is, $M$ is finite-type and free if and only if there exists a function $c: \mathbb{Z} \rightarrow \mathbb{N}$
and an isomorphism of graded $A$-modules 
 \[ \coprod_{n\in \mathbb{Z}} (\Sigma^n A)^{\oplus c(n)} \stackrel{\cong}{\longrightarrow} M.\]
\item
We will say that a graded $A$-module $M$ has {\em finite-type generators} if $M$ admits a set of homogeneous generators, with finitely many in each degree. That is, $M$ has finite-type generators if and only if there exists a short exact sequence of graded $A$-modules
\begin{equation}\label{exact seq 5} F_1 \rightarrow F_0 \rightarrow M \rightarrow 0\end{equation}
with $F_0$ finite-type and free.  
\item 
We will say that $M$ is {\em finite-type} if $M$ admits a presentation given by homogeneous generators, finitely many in each degree, and homogeneous relations, finitely many in each degree. That is $M$ is finite-type if and only if there exists an exact sequence of graded $A$-modules as in~\eqref{exact seq 5}, with $F_0,F_1$ both finite-type and free.
\end{itemize}
\end{definition}

Lemmas~\ref{finite type generators lemma} and~\ref{finite type lemma} are generalizations, to the graded setting, of two useful lemmas found in T.~Y.~Lam's book~\cite{MR1653294}. The ungraded versions of these lemmas appear as Propositions~2.4.43 and~2.4.44 in Lam's book. We provide proofs of Lemmas~\ref{finite type generators lemma} and~\ref{finite type lemma} for the sake of being self-contained, but there is nothing novel here: the proofs are essentially the same as in the ungraded case.
I am grateful to G. Valenzuela for suggesting Lam's book to me as a reference for the ungraded results. 
\begin{lemma}\label{finite type generators lemma}
Let $A$ be an connective graded ring and let $\Gamma$ be a connective graded left $A$-module. 
The following conditions are equivalent:
\begin{itemize}
\item For every set $\{ M_i\}_{i\in I}$ of connective graded left $A$-modules,
the canonical graded $A$-module map
\begin{equation}\label{comparison map 10} \Gamma\otimes_A \prod_{i\in I} M_i \rightarrow \prod_{i\in I} (\Gamma\otimes_A M_i) \end{equation}
is surjective.
\item For every set $I$, the canonical graded $A$-module map
\begin{equation}\label{comparison map 11} \Gamma\otimes_A \prod_{i\in I} A \rightarrow \prod_{i\in I} \Gamma \end{equation}
is surjective.
\item As a graded $A$-module, $\Gamma$ has finite-type generators.
\end{itemize}
\end{lemma}
\begin{proof}
\begin{itemize}
\item If the first condition is satisfied, then letting $M_i = A$ for all $i\in I$, we immediately get that the second condition is satisfied.
\item Suppose that the second condition is satisfied. 
Choose an integer $n$, and let
$I$ be the set of homogeneous elements of $\Gamma$ of grading degree exactly $n$. We will write $\prod_{i\in I} \Gamma\{ e_i\}$ for the product
$\prod_{i\in I} \Gamma$, using $e_i$ as formal symbols to index the factors in the product. Let $x_n\in \prod_{i\in I} \Gamma\{ e_i\}$ be the element $x_n = \sum_{i\in I} i \cdot e_{i}.$
Since the map~\eqref{comparison map 11} is grading-preserving and surjective, there exists some
element \[ \sum_{j=1}^{m_n} \left( c_{j,n} \otimes \sum_{i\in I} a_{i,j,n}e_{i}\right) \in \Gamma\otimes_A \prod_{i\in I} A\{ e_{i}\}\]
which is sent by the map~\eqref{comparison map 11} to $x_n$,
in which each $c_{j,n}$ is a homogeneous element of $\Gamma$ and in which
each $a_{i,j,n}$ is a homogeneous element of $A$.
The grading degrees of these elements satisfy
$\left| c_{j,n}\right| + \left| a_{i,j,n} \right| = n$,
and consequently $\left| c_{j,n}\right| \leq n$.

Consequently we have the formula
\[ \sum_{j=1}^{m_n}\sum_{i\in I} c_{j,n} a_{i,j,n}e_{i} = \sum_{i\in I} i \cdot e_{i},\]
and consequently $\sum_{j=1}^{m_n}  c_{j,n}a_{i, j, n} = i$. 
Consequently the set of elements $S = \{ c_{j,n}: n\in \mathbb{Z}, 1\leq j\leq m_n\}$ is
a set of homogeneous $A$-module generators for $\Gamma$. 
Let $S_n$ be the set $\{ c_{j,n}: 1\leq j\leq m_n\}\subseteq \Gamma$, so that
$S = \bigcup_{n\in\mathbb{Z}} S_n$. 
Then each $S_n$ is finite, and, given an element of $S$ in grading degree $N$,
that element must be contained in $S_n$ for some $n\leq N$, of which there are only finitely many, since $\Gamma$ is connective.
So, for each integer $N$, there are only
finitely many elements of $S$ of grading degree $\leq N$.
Hence there are only finitely many
elements of $S$ in each grading degree. Hence $\Gamma$ has finite-type generators. 
\item 
Now suppose that $\Gamma$ has finite-type generators, and that 
$\{ M_i\}_{i\in I}$ is a set of graded left $A$-modules.
We need to show that map~\eqref{comparison map 10} is surjective.

Choose a set of homogeneous $A$-module generators
$\{ c_{j}\}_{j\in J}$
for $\Gamma$, with at most finitely many $c_j$ in each grading degree.
Let $D: J \rightarrow \mathbb{Z}$ be the function that sends
$j$ to the grading degree of $c_j$.
For each integer $n$, let $\Gamma_{\leq n}$ be the graded sub-$A$-module
of $\Gamma$ generated by all the elements $c_j$ such that $D(j)\leq n$.
Since $A$ is connective and all $M_i$ are connective, the natural map $\Gamma_{\leq n}\hookrightarrow \Gamma$ of graded $A$-modules is bijective in grading degrees $\leq n$.

Write $J_n$ for the set of elements $j\in J$ such that $D(j)\leq n$.
Now we have an exact sequence of $A$-modules
\[ \coprod_{j\in J_n} \Sigma^{D(j)}A\{ e_j\} \stackrel{s}{\longrightarrow} \Gamma_{\leq n} \rightarrow 0\]
where $s(e_j) = c_j$; here the elements $e_j$ are formal symbols
indexing the coproduct summands.
The map $s$ now fits into the commutative square of graded $A$-modules
\begin{equation}\label{diagram 4} \xymatrix{
 \left( \coprod_{j\in J_n} \Sigma^{D(j)}A\{ e_j\}\right) \otimes_A\prod_{i\in I}M_i  \ar[r]^>>>>>{s\otimes \id} \ar[d] & 
  \Gamma_{\leq n}\otimes_A \prod_{i\in I}M_i \ar[d]  \\
 \prod_{i\in I}\left(  \left( \coprod_{j\in J_n} \Sigma^{D(j)}A\{ e_j\}\right) \otimes_A M_i \right) \ar[r]_>>>>>{\prod s\otimes \id} &
  \prod_{i\in I}\left( \Gamma_{\leq n} \otimes_A M_i \right)
}\end{equation}
where the vertical maps are the canonical comparison maps, as in map~\eqref{comparison map 10}. The map $\prod s\otimes \id$ is a surjection, 
since each $s\otimes \id$ is a surjection and since
infinite direct products are exact in the category of graded $A$-modules.
The left-hand vertical map in diagram~\eqref{diagram 4} 
is an isomorphism, since $J_n$ is finite. Hence the right-hand vertical map in diagram~\eqref{diagram 4} is also surjective. 
The square of graded $A$-modules
\begin{equation}\label{diagram 5} \xymatrix{
 \Gamma_{\leq n}\otimes_A \prod_{i\in I}M_i \ar[d] \ar[r] &
  \Gamma \otimes_A \prod_{i\in I}M_i \ar[d] \\
 \prod_{i\in I}\left( \Gamma_{\leq n} \otimes_A M_i \right) \ar[r] &
  \prod_{i\in I}\left( \Gamma \otimes_A M_i \right)
}\end{equation}
commutes, and the horizontal maps are isomorphisms in grading degrees $\leq n$, so surjectivity of the right-hand vertical map in diagram~\eqref{diagram 4}, i.e., the left-hand vertical map in diagram~\eqref{diagram 5}, tells us that the right-hand vertical map in diagram~\eqref{diagram 5}, i.e., the map~\eqref{comparison map 10}, is surjective in grading degree $n$.
But this holds for all integers $n$; so the map~\eqref{comparison map 10} is surjective.
\end{itemize}
\end{proof}

\begin{lemma}\label{finite type lemma}
Let $A$ be an connective graded ring and let $\Gamma$ be a connective graded left $A$-module. 
The following conditions are equivalent:
\begin{itemize}
\item For every set $\{ M_i\}_{i\in I}$ of connective graded left $A$-modules,
the canonical graded $A$-module map
\begin{equation}\label{comparison map 10a} \Gamma\otimes_A \prod_{i\in I} M_i \rightarrow \prod_{i\in I} (\Gamma\otimes_A M_i) \end{equation}
is an isomorphism.
\item For every set $I$, the canonical graded $A$-module map
\begin{equation}\label{comparison map 11a} \Gamma\otimes_A \prod_{i\in I} A \rightarrow \prod_{i\in I} \Gamma \end{equation}
is an isomorphism.
\item As a graded $A$-module, $\Gamma$ is finite-type.
\end{itemize}
\end{lemma}
\begin{proof}
\begin{itemize}
\item If the first condition is satisfied, then letting $M_i = A$ for all $i\in I$, we immediately get that the second condition is satisfied.
\item Suppose that the second condition is satisfied. 
We will write $\prod_{i\in I} \Gamma\{ e_i\}$ for the product
$\prod_{i\in I} \Gamma$, using $e_i$ as formal symbols to index the factors in the product.

By Lemma~\ref{finite type generators lemma}, we know that $\Gamma$ has finite-type generators. 
Choose an exact sequence of graded $A$-modules
\begin{equation}\label{exact seq 6} 0 \rightarrow K \rightarrow F_0 \rightarrow \Gamma \rightarrow 0\end{equation}
with $F_0$ finite-type and free.  We can arrange maps as in~\eqref{comparison map 11a} into a commutative diagram with exact rows
\[ \xymatrix{
  & (\prod_{i\in I} A) \otimes_A K \ar[r]\ar[d] & (\prod_{i\in I} A) \otimes_A F_0 \ar[r]\ar[d] & (\prod_{i\in I} A) \otimes_A \Gamma \ar[r]\ar[d] & 0\ar[d] \\
 0 \ar[r] & \prod_{i\in I} K \ar[r] & \prod_{i\in I} F_0 \ar[r] & \prod_{i\in I} \Gamma \ar[r] & 0
}\]
in which the vertical map 
$(\prod_{i\in I} A) \otimes_A \Gamma\rightarrow \prod_{i\in I}\Gamma$ is an isomorphism by assumption, and the vertical map 
$(\prod_{i\in I} A) \otimes_A F_0 \rightarrow \prod_{i\in I} F_0$ is surjective by
Lemma~\ref{finite type generators lemma}. An easy diagram chase
shows that the vertical map 
$(\prod_{i\in I} A) \otimes_A K \rightarrow \prod_{i\in I} K$ is then 
also surjective. By Lemma~\ref{finite type generators lemma}, $K$ then has finite-type generators, hence we can choose a finite-type and free graded $A$-module $F_1$ and a
surjective graded $A$-module map $F_1\rightarrow K$, and consequently
\[ F_1 \rightarrow F_0 \rightarrow \Gamma \rightarrow 0\]
is an exact sequence of graded $A$-modules with $F_1,F_0$ finite-type and free.
So $\Gamma$ is finite-type.
\item
Now suppose that $\Gamma$ is finite-type.
First, suppose that $\Gamma$ is finite-type and free.
Choose a set of homogeneous $A$-module generators $S$ for $\Gamma$ with 
at most finitely many elements of $S$ in each grading degree, and 
then let $\Gamma_{\leq n}$ be the graded sub-$A$-module of $\Gamma$
generated by the elements of $S$ of degree $\leq n$.
Since $A$ and $\Gamma$ and all $M_i$ are connective, 
the horizontal maps in the commutative square
\begin{equation}\label{comm sq 14} \xymatrix{
 \Gamma_{\leq n}\otimes_A \prod_{i\in I} M_i  \ar[r] \ar[d]^{} & \Gamma\otimes_A \prod_{i\in I} M_i  \ar[d] \\
 \prod_{i\in I}\left(\Gamma_{\leq n}\otimes_A  M_i\right)  \ar[r] & \prod_{i\in I}\left(\Gamma\otimes_A  M_i \right)
}\end{equation}
are isomorphisms in grading degrees $\leq n$,
and the left-hand vertical map
is an isomorphism in grading degrees $\leq n$, since $\Gamma_{\leq n}$ is a direct sum of finitely many
copies of $A$ (up to suspension), and finite direct sums coincide with
finite products in module categories, including graded module categories.
Consequently the right-hand vertical map in square~\eqref{comm sq 14} is also
an isomorphism in grading degrees $\leq n$.
Since this is true for all $n$, the canonical map~\eqref{comparison map 10a} is an isomorphism when $\Gamma$ is finite-type and free.

Now lift the assumption that $\Gamma$ is finite-type and free, and assume it is only finite-type.
Choose an exact sequence
of graded $A$-modules
\[ F_1 \rightarrow F_0 \rightarrow \Gamma \rightarrow 0\]
with $F_1,F_0$ finite-type and free.
We can fit maps as in~\eqref{comparison map 10a}
into the commutative diagram of graded $A$-modules with exact rows
\[\xymatrix{
  F_1 \otimes_A \prod_{i\in I} M_i \ar[r]\ar[d] &
  F_0 \otimes_A \prod_{i\in I} M_i \ar[r]\ar[d] &
  \Gamma \otimes_A \prod_{i\in I} M_i \ar[r] \ar[d] &
  0 \ar[d] \\
   \prod_{i\in I} F_1\otimes_A M_i \ar[r] & 
  \prod_{i\in I} F_0\otimes_A M_i \ar[r] & 
  \prod_{i\in I} \Gamma \otimes_A M_i \ar[r] & 
  0 
}\]
and the two left-hand vertical maps are both isomorphisms, by what we have already proven under the finite-type-and-free assumption; hence the map
$\Gamma \otimes_A \prod_{i\in I} M_i \rightarrow \prod_{i\in I} \Gamma \otimes_A M_i $
is an isomorphism.
\end{itemize}
\end{proof}

\section{Graded comodules.}

\begin{definition}
Let $(A,\Gamma)$ be a graded Hopf algebroid.
We will say that a graded $\Gamma$-comodule $M$ is
{\em finite-type} if $M$ is finite-type as an $A$-module, as in Definition~\ref{def of finite-type}. 
We will say that the graded Hopf algebroid $(A,\Gamma)$ is itself
{\em finite-type} if $\Gamma$ is finite-type as an $A$-module.

Similarly, we will say that a comodule is {\em connective} if it is connective as an graded $A$-module. We will say that the Hopf algebroid $(A,\Gamma)$ is {\em connective} if $A$ and $\Gamma$ are both connective as graded $A$-modules.
\end{definition}

\begin{example}
The graded Hopf algebroid $(MU_*,MU_*MU)$ satisfies $MU_* \cong \mathbb{Z}[x_1, x_2, \dots ]$ and $MU_*MU\cong MU_*[b_1, b_2, \dots ]$, with $\left| x_n\right| = \left| b_n\right| = 2n$, so $(MU_*,MU_*MU)$ is flat, connective, and finite-type.
Similarly, $BP_* \cong \mathbb{Z}_{(p)}[v_1, v_2, \dots ]$ and $BP_*BP \cong BP_*[t_1, t_2, \dots ]$ with $\left| v_n\right| = \left| t_n\right| = 2(p^n-1)$ for a given prime number $p$ (the choice of $p$ is suppressed from the notation for $BP$), so $(BP_*,BP_*BP)$ is flat, connective, and finite-type.
Finally, $(H\mathbb{F}_p)_* \cong \mathbb{F}_p$, and $(H\mathbb{F}_p)_*H\mathbb{F}_p \cong \mathbb{F}_2[\xi_1, \xi_2, \dots ]$ if $p=2$, with $\left| \xi_n\right| = 2^n-1$; and $(H\mathbb{F}_p)_*H\mathbb{F}_p \cong \mathbb{F}_p[\xi_1, \xi_2, \dots ]\otimes_{\mathbb{F}_p} \Lambda(\tau_0,\tau_1, \dots)$ if $p>2$, with $\left| \xi_n\right| = 2(p^n-1)$ and $\left| \tau_n\right| = 2p^n-1$, so again, $((H\mathbb{F}_p)_*, (H\mathbb{F}_p)_*H\mathbb{F}_p)$ is flat, connective, and finite-type.
See chapters~3 and ~4 of~\cite{MR860042} for this material (which is well-known in homotopy theory).

For another class of examples: $(k, A)$ is flat, connective, and finite-type for any commutative graded connected finite-type Hopf algebra $A$ over a field $k$ (as studied in~\cite{MR0174052}).
\end{example}

\begin{lemma}\label{products lemma}
Let $(A,\Gamma)$ be a connective finite-type flat graded Hopf algebroid. 
Let $\{ M_i\}_{i\in I}$ be a set of connective graded $\Gamma$-comodules. 
Then the natural map of graded $A$-modules
\begin{equation}\label{comparison map} \prod^{\Gamma}_{i\in I} M_i \rightarrow \prod_{i\in I} M_i,\end{equation}
from the underlying graded $A$-module of the product of the $M_i$ computed in the category of connective graded $\Gamma$-comodules to the product of the $M_i$ computed in the category of graded $A$-modules, is an isomorphism.
\end{lemma}
\begin{proof}
I am grateful to the anonymous referee for pointing out that this lemma follows from Lemma \ref{finite type lemma}, above, together with a general result about limits in the category of coalgebras over a comonad, which one can find as (the dual to) Proposition 4.3.2 in \cite{MR1313497}. I include a self-contained proof here as well, because the proof is short, direct, and (in my opinion) illuminating. Write $\gr_{\geq 0}\mathcal{C}$ for the connective graded objects in an abelian category $\mathcal{C}$.
Write $G: \gr_{\geq 0}\Comod(\Gamma)\rightarrow \gr_{\geq 0}\Mod(A)$ for the forgetful functor
and $E: \gr_{\geq 0}\Mod(A)\rightarrow\gr_{\geq 0}\Comod(\Gamma)$ for its right adjoint, the extended comodule functor given by $E(M) = \Gamma\otimes_A M$.

For each $i\in I$, we have the exact sequence
\[ 0 \rightarrow M_i \rightarrow EG(M_i)\stackrel{\delta^0}{\longrightarrow} EGEG(M_i)\]
of graded $\Gamma$-comodules, where $\delta^0$ is the difference of the two unit maps
arising from the adjunction $G\dashv E$. (This is well-known; it is the reason that the cobar
resolution of a comodule is indeed a resolution, as in Appendix~1 of~\cite{MR860042}. The reader who prefers
a self-contained, categorical argument may be satisfied with the observation that, for any adjunction 
$f \dashv g$, the cofork
\begin{equation}\label{}\xymatrix{
X \ar[r] & 
 gfX \ar@<1ex>[r]\ar@<-1ex>[r]
   & 
  gfgfX 
  }\end{equation}
splits after applying $f$; see section VI.6 of~\cite{MR1712872}. But in our setting, $f = G$, the left adjoint functor $f$ reflects isomorphisms,
so the canonical map $X \rightarrow \ker \delta^0$ being an isomorphism after applying $f$, due to the splitting of the cofork, 
implies that $$X \rightarrow \ker \delta^0$$ is already an isomorphism.)

Now the fact that products preserve kernels tells us that 
we have the commutative diagram with exact rows
\begin{equation}\label{diagram 116} \xymatrix{
  &
  &
 GE\left(\prod_{i\in I} GM_i\right) \ar[r]\ar[d]^{\cong} &
 GE\left(\prod_{i\in I} GEGM_i\right) \ar[d]^{\cong}  \\
0 \ar[r]\ar[d] &  
 G\left(\prod^{\Gamma}_{i\in I} M_i\right) \ar[r]\ar[d] &
 G\left(\prod^{\Gamma}_{i\in I} EGM_i\right) \ar[r]\ar[d]^{} &
 G\left(\prod^{\Gamma}_{i\in I} EGEGM_i\right) \ar[d]^{} \\
0 \ar[r] &
 \prod_{i\in I} G(M_i) \ar[r] &
 \prod_{i\in I} GEG(M_i) \ar[r] &
 \prod_{i\in I} GEGEG(M_i).
}\end{equation}
The maps indicated as isomorphisms are isomorphisms due to $E$ being a right adjoint, hence preserving products. The vertical composites
$GE\left(\prod_{i\in I} GM_i\right) \rightarrow  \prod_{i\in I} GEG(M_i)$
and 
$GE\left(\prod_{i\in I} GEGM_i\right) \rightarrow  \prod_{i\in I} GEGEG(M_i)$
are the maps
$\Gamma\otimes_A \prod_{i\in I} M_i \rightarrow  \prod_{i\in I} \Gamma\otimes_A M_i$
and 
$\Gamma\otimes_A \prod_{i\in I} \Gamma\otimes_A M_i \rightarrow  \prod_{i\in I} \Gamma\otimes_A\Gamma \otimes_A M_i$, respectively, 
of the type~\eqref{comparison map 10a}. Lemma~\ref{finite type lemma}
then implies that these maps are isomorphisms.
Consequently the map
$G\left(\prod^{\Gamma}_{i\in I} M_i\right) \rightarrow \prod_{i\in I} G(M_i)$
in diagram~\eqref{diagram 116}
is an isomorphism.
\end{proof}

We now give a sequence of lemmas which refer to generators, cogenerators, and compactness.
Recall that, given an abelian category $\mathcal{C}$, an object $M$ of $\mathcal{C}$ is said to be {\em compact} if the functor $\hom_{\gr_{\geq 0}\Comod(\Gamma)}(M,-): \gr_{\geq 0}\Comod(\Gamma) \rightarrow \Ab$ commutes with filtered colimits; and $M$ is said to be a {\em generator} if the functor $\hom_{\gr_{\geq 0}\Comod(\Gamma)}(M,-)$ is faithful. ``Cogenerator'' is defined dually to ``generator.''
\begin{lemma}\label{cogenerator lemma}
Let $(A,\Gamma)$ be a flat graded Hopf algebroid. Suppose that $A$ is connective.
Then the category of connective graded $\Gamma$-comodules is abelian and has an injective cogenerator.
\end{lemma}
\begin{proof}
Let $\gr_{\geq 0}\Comod(\Gamma)$ denote the category of connective graded $\Gamma$-comodules, let 
$\gr\Comod(\Gamma)$ denote the category of
graded $\Gamma$-comodules,
and let $\gr_{\geq 0}\Mod(A)$ denote the category of connective graded $A$-modules.
It is standard that $\gr\Comod(\Gamma)$ is abelian
as long as $\Gamma$ is flat over $A$; see Theorem~1.1.3 of~\cite{MR860042}, for example. 
Since $\gr_{\geq 0}\Comod(\Gamma)$ is a full additive subcategory of $\gr\Comod(\Gamma)$ which is closed under finite biproducts and kernels and cokernels computed in
$\gr\Comod(\Gamma)$, the category $\gr_{\geq 0}\Comod(\Gamma)$ is abelian as well; see Theorem~3.41 of~\cite{MR2050440}, for example.

Now let $E: \gr_{\geq 0}\Mod(A)\rightarrow \gr_{\geq 0}\Comod(\Gamma)$ be the extended comodule functor.
The idea here is to apply $E$ to a cogenerator in the category of graded $A$-modules, but if $A$ is not concentrated in a single grading degree, then a cogenerator for the category of graded $A$-modules will typically fail to be connective, so applying $\Gamma\otimes_A -$ to such a cogenerator does not yield a connective graded comodule.

Instead, we will apply $E$ to an injective cogenerator $I$ in the category $\gr_{\geq 0}\Mod(A)$ of {\em connective} graded $A$-modules---but we must show that $I$ exists.
Since kernels and colimits in $\gr_{\geq 0}\Mod(A)$ are computed
in the underlying category of graded $A$-modules, and since graded $A$-modules form an AB$5$ abelian category, the category $\gr_{\geq 0}\Mod(A)$ is also AB$5$.
The coproduct $\coprod_{n\geq 0} \Sigma^n A$ is a generator for $\gr_{\geq 0}\Mod(A)$,
so $\gr_{\geq 0}\Mod(A)$ is Grothendieck, so by Grothendieck's famous theorem in~\cite{MR0102537} (that every Grothendieck category has an injective cogenerator), $\gr_{\geq 0}\Mod(A)$ has an injective cogenerator. So $I$ exists.

Now the functor $E$ is right adjoint to the forgetful functor $G: \gr_{\geq 0}\Comod(\Gamma)\rightarrow\gr_{\geq 0}\Mod(A)$, and $G$ preserves monomorphisms since kernels of comodule maps are computed in the underlying module category; it is an elementary exercise to show that a functor sends injectives to injectives if it has a monomorphism-preserving left adjoint. So $E(I)$ is an injective object in connective graded $\Gamma$-comodules.
We claim that $E(I)$ is also a cogenerator. Let $f: X \rightarrow Y$ be a morphism in $\gr_{\geq 0}\Comod(\Gamma)$ whose induced map
\[ \hom_{\gr_{\geq 0}\Comod(\Gamma)}(Y, E(I)) \rightarrow \hom_{\gr_{\geq 0}\Comod(\Gamma)}(X, E(I))\] is zero. 
Then the adjunction $G\dashv E$ tells us that the map
\[\hom_{\gr_{\geq 0}\Mod(A)}(G(Y), I) \rightarrow \hom_{\gr_{\geq 0}\Mod(A)}(G(X), I)\]
is zero, and hence that $G(f): G(X) \rightarrow G(Y)$ is zero, since $I$ is a cogenerator in $\gr_{\geq 0}\Mod(A)$. 
Since $G$ is faithful and additive, this then tells us that $f=0$. So $E(I)$ is an injective cogenerator in $\gr_{\geq 0}\Comod(\Gamma)$.
\end{proof}

\begin{lemma}\label{E commutes with colimits}
Let $(A,\Gamma)$ be a connective graded flat Hopf algebroid. Then the extended comodule functor $E: \gr_{\geq 0}\Mod(A)\rightarrow \gr_{\geq 0}\Comod(\Gamma)$ commutes with all colimits.
\end{lemma}
\begin{proof}
Let $\mathcal{A}$ be a small category and let $H: \mathcal{A}\rightarrow \gr_{\geq 0}\Mod(A)$ be a functor. We continue to write $G$ for the forgetful functor $\gr_{\geq 0}\Comod(\Gamma)\rightarrow\gr_{\geq 0}\Mod(A)$ which is left adjoint to $E$. The composite $GE$ is $\Gamma\otimes_A -$, hence preserves colimits in $\gr_{\geq 0}\Mod(A)$. So the composite natural map
\begin{equation*} \colim_{d\in D} GEH(d)  \stackrel{\cong}{\longrightarrow} G\colim_{d\in D} EH(d)  \rightarrow GE\colim_{d\in D} H(d) \end{equation*}
is an isomorphism, so the comparison map $\colim_{d\in D} EH(d)  \rightarrow E\colim_{d\in D} H(d)$ is an isomorphism after applying $G$, hence is already an isomorphism since $G$ reflects isomorphisms.
\end{proof}

\begin{lemma}\label{functors preserving compact objects}
Given abelian categories $\mathcal{C},\mathcal{D}$, a compact object $M$ of $\mathcal{C}$, and a functor $F: \mathcal{C}\rightarrow \mathcal{D}$ with right adjoint $G$ such that $G$ preserves filtered colimits, the object $F(M)$ of $\mathcal{D}$ is compact.
\end{lemma}
\begin{proof}
Elementary exercise in applying adjunctions.
\end{proof}

\begin{lemma}\label{compacts are fg}
Let $A$ be a graded ring. If a graded $A$-module $M$ is a compact object in the category $\gr_{\geq 0}\Mod(A)$ of connective graded $A$-modules, then $M$ is finitely generated.
\end{lemma}
\begin{proof}
A standard exercise: writing $\{ M_i\}_{i\in I}$ for the filtered collection (ordered by inclusion) of finitely generated graded sub-$A$-modules of $M$, we have that the
map
\begin{align*} \colim_{i\in I} \hom_{\gr_{\geq 0}\Mod(A)}(M, M_i) &\rightarrow \hom_{\gr_{\geq 0}\Mod(A)}(M,\colim_i M_i) \\ &\cong \hom_{\gr_{\geq 0}\Mod(A)}(M,M) \end{align*}
is an isomorphism, and consequently that the identity map on $M$ factors through some $M_i$, i.e., $M$ is a summand in a finitely generated graded $A$-module, so $M$ is itself a finitely generated graded $A$-module.
\end{proof}

\begin{theorem}\label{main thm}
Let $(A,\Gamma)$ be a connective finite-type graded flat Hopf algebroid.
Then the category of connective graded $\Gamma$-comodules
is a Grothendieck category with a projective generator. 
Consequently, 
the category of connective graded $\Gamma$-comodules
has enough projectives and enough injectives, and satisfies Grothendieck's axiom
AB$4^*$ (that is, infinite products exist and are exact).
\end{theorem}
\begin{proof}
By Lemma~\ref{cogenerator lemma}, $\gr_{\geq 0}\Comod(\Gamma)$ is abelian and has an injective cogenerator. (This would also be implied by $\gr_{\geq 0}\Comod(\Gamma)$ being a Grothendieck category, but at this point in this proof, we are still on our way to proving that $\gr_{\geq 0}\Comod(\Gamma)$ is Grothendieck.)
By Lemma~\ref{products lemma}, products in $\gr_{\geq 0}\Comod(\Gamma)$ are computed in $\gr_{\geq 0}\Mod(A)$,
hence products in $\gr_{\geq 0}\Comod(\Gamma)$ are exact, since the
category of graded modules over any ring is AB$4^*$.
So $\gr_{\geq 0}\Comod(\Gamma)$ satisfies axiom AB$4^*$. 
(In any Grothendieck category, having enough projectives implies that
the category satisfies Grothendieck's axiom AB$4^*$---see Corollary~1.4 of~\cite{MR2197371} for a proof---but the converse is
{\em not} true: see~\cite{MR2197371} for examples, due to Gabber and Roos,
of Grothendieck categories satisfying axiom AB$4^*$ but having no nonzero projectives at all!)

More precisely, Lemma~\ref{products lemma} shows that the forgetful functor $G: \gr_{\geq 0}\Comod(\Gamma)\rightarrow\gr_{\geq 0}\Mod(A)$ preserves products. In this paragraph and the next two, we show that $G$ preserving products is the key result which causes $\gr_{\geq 0}\Comod(\Gamma)$ to have enough projectives.
The functor $G$ is also easily seen to preserve
kernels (see e.g. Appendix 1 of~\cite{MR860042} for the usual construction of kernels in graded $\Gamma$-comodules; the salient point is that they are computed in the underlying category of graded $A$-modules), so $G$ preserves all limits. 
Now  $\gr_{\geq 0}\Comod(\Gamma)$ is certainly ``well-powered,'' that is, every connective graded $\Gamma$-comodule has only a set (not a proper class) of subcomodules; and 
by Lemma~\ref{cogenerator lemma}, $\gr_{\geq 0}\Comod(\Gamma)$ has a cogenerator.
So by Freyd's Special Adjoint Functor Theorem (standard; see e.g. Theorem~V.8.2 of~\cite{MR1712872}, or for a statement closer to our application here, section 3.M of~\cite{MR2050440}), $G$ has a left adjoint. 
Call this left adjoint $F$. Since $F$ has a right adjoint (namely, $G$) which preserves epimorphisms, $F$ sends projectives to projectives. 
So $F(\coprod_{n\geq 0}\Sigma^n A)$ is a projective object of $\gr_{\geq 0}\Comod(\Gamma)$, since $\coprod_{n\geq 0} \Sigma^n A$ is projective in $\gr_{\geq 0}\Mod(A)$.

We claim that $F(\coprod_{n\geq 0}\Sigma^n A)$ is also a generator in $\gr_{\geq 0}\Comod(\Gamma)$.
The proof is as follows: if $V$ is a generator of $\gr_{\geq 0}\Mod(A)$ and $f: X \rightarrow Y$ a map in $\gr_{\geq 0}\Comod(\Gamma)$
whose induced map \[ \hom_{\gr_{\geq 0}\Comod(\Gamma)}(FV, X) \rightarrow \hom_{\gr_{\geq 0}\Comod(\Gamma)}(FV, Y)\] is zero, then the adjunction $F\dashv G$
gives us that the induced map \[\hom_{\gr_{\geq 0}\Mod(A)}(V, GX) \rightarrow \hom_{\gr_{\geq 0}\Mod(A)}(V, GY)\] is zero and hence that
$Gf: GX \rightarrow GY$ is zero. Since $G$ is faithful and additive, $f = 0$. So $FV$ is a generator in $\gr_{\geq 0}\Comod(\Gamma)$.

Consequently $\gr_{\geq 0}\Comod(\Gamma)$ is a cocomplete abelian category with a projective generator. It is standard that this now implies that $\gr_{\geq 0}\Comod(\Gamma)$ has enough projectives: if $\mathcal{C}$ is a cocomplete abelian category with projective generator $P$, then for any object $X$ of $\mathcal{C}$, 
the object $\coprod_{f\in\hom_{\mathcal{C}}(P,X)}P$ is projective,
and the evaluation map $\coprod_{f\in\hom_{\mathcal{C}}(P,X)}P \rightarrow X$ is epic.

Since $\gr_{\geq 0}\Mod(A)$ satisfies Grothendieck's axiom AB$5$ (see the proof of Lemma~\ref{cogenerator lemma} for this), and since $G$ is faithful, additive, has both a left and a right adjoint and hence is exact and preserves all colimits, 
$\gr_{\geq 0}\Comod(\Gamma)$ also satisfies Grothendieck's axiom AB$5$. 
So $\gr_{\geq 0}\Comod(\Gamma)$ satisfies AB$5$ and has a generator, hence
$\gr_{\geq 0}\Comod(\Gamma)$ is Grothendieck.
\end{proof}
There is a classical ``recognition principle'' for the category of modules over a ring (see Corollary~V.1 of~\cite{MR0232821}): an abelian category is equivalent to the category of modules over a ring if and only if that abelian category is cocomplete and has a {\em compact} projective generator. Theorem \ref{main thm} tells us that the category $\gr_{\geq 0}\Comod(\Gamma)$ of connective comodules over a connective finite-type graded flat Hopf algebroid is co-complete and has a projective generator. This seems, at a glance, like it is awfully close to saying that $\gr_{\geq 0}\Comod(\Gamma)$ is equivalent to the category of modules over a ring. However, the projective generator for $\gr_{\geq 0}\Comod(\Gamma)$ that we construct in the proof of Theorem \ref{main thm} is infinitely generated, hence far from being compact. One might ask if it is possible to find a smaller projective generator, one which {\em is} compact. We now give the simple argument for why, except in trivial cases, this is impossible:
\begin{prop}\label{gr-comod isnt a module category}
Let $(A,\Gamma)$ be as in Theorem \ref{main thm}. If $A$ is not the zero ring, then $\gr_{\geq 0}\Comod(\Gamma)$ is {\em not} equivalent to the category of modules over a ring. 
\end{prop}
\begin{proof}
Suppose that $M$ is a compact generator for $\gr_{\geq 0}\Comod(\Gamma)$. 
If we assume that the underlying graded $A$-module of $M$ admits a set of homogeneous generators concentrated in finitely many grading degrees, then we get a contradiction as follows: let $S$ denote a minimal set of homogeneous generators for the underlying $A$-module of $M$, and let $n$ be an upper bound for the grading degrees of the elements of $S$.
Every map of graded $\Gamma$-comodules $M \rightarrow \Sigma^{n+1} A$ must send all $A$-module generators of $M$ to zero, so the functor $\hom_{\gr_{\geq 0}\Comod(\Gamma)}(M,-)$ fails to distinguish between the zero map $\Sigma^{n+1} A \rightarrow \Sigma^{n+1} A$ and the identity map on $\Sigma^{n+1} A$, contradicting faithfulness of $\hom_{\gr_{\geq 0}\Comod(\Gamma)}(M,-)$. (The previous sentence is where we have used the assumption $A\neq 0$.) 

So, if we choose a set of homogeneous generators $\{ m_i\}_{i\in I}$ for the underlying $A$-module of $M$, there must be elements $m_i$ in arbitrarily high grading degrees. In particular, the underlying $A$-module of $M$ is not finitely generated, consequently not compact by Lemma~\ref{compacts are fg}. But applying Lemmas~\ref{E commutes with colimits}, \ref{functors preserving compact objects}, and \ref{compacts are fg}, $G(M)$ is a finitely generated graded $A$-module, a contradiction. So $M$ must not exist. 
\end{proof}
\begin{corollary}\label{main cor}
The categories of connective graded comodules over the Hopf algebroids
$(MU_*,MU_*MU)$, $(BP_*,BP_*BP)$, and $((H\mathbb{F}_p)_*, (H\mathbb{F}_p)_*H\mathbb{F}_p)$ all have enough projectives. None of these categories is equivalent to the category of modules over a ring.
\end{corollary}

Perhaps the statement of Proposition \ref{gr-comod isnt a module category} sounds a bit strange: after all, if $\Gamma=A$, then the category of graded $\Gamma$-comodules is simply the category of graded $A$-modules. The reason that Proposition \ref{gr-comod isnt a module category} works is that the category of connective {\em graded} modules over a connective graded ring $A$ is not equivalent to the category of {\em ungraded} modules over a ring. So what makes Proposition \ref{gr-comod isnt a module category} work is not really specific to comodules at all: it is essentially the same phenomenon which is responsible for {\em graded} module categories being inequivalent to {\em ungraded} module categories. 

Consequently, we ought to show that $\gr_{\geq 0}\Comod(\Gamma)$ is not equivalent to the category of connective {\em graded} modules over a graded ring. This takes a bit more work. The purpose of the next section is to do this in the simplest and most classical nontrivial case, when $\Gamma$ is a graded commutative Hopf algebra over a field. (This is of course a very commonly-occurring case: it occurs when $\Gamma$ is the dual Steenrod algebra at any prime, for example.)
 
\section{The case of a Hopf algebra over a field.}

Throughout this section, we restrict our attention to the situation where the Hopf algebroid $(A,\Gamma)$ is a connective graded commutative Hopf algebra over a field. Consequently $A$ will be a field in this section, and to reinforce this running assumption, we change notation slightly, and write $k$ in place of $A$.

\subsection{Calculation of the $k$-vector space underlying a generator for the category of connective graded $\Gamma$-comodules.}

Recall from Theorem \ref{main thm} that we constructed a left adjoint $F: \gr_{\geq 0}\Mod(k)\rightarrow\gr_{\geq 0}\Comod(\Gamma)$ to the forgetful functor $G: \gr_{\geq 0}\Comod(\Gamma)\rightarrow\gr_{\geq 0}\Mod(k)$. The effect of $F$ on suspensions $\Sigma^n k$ of the ground ring was especially important in the rest of the proof of Theorem \ref{main thm}, since we showed that $\coprod_{n\geq 0} F(\Sigma^n k)$ is a projective generator for the category $\gr_{\geq 0}\Comod(\Gamma)$.
Our first task in this section is to give a more concrete identification of the $\Gamma$-comodule $F\Sigma^n k$. 
Proposition \ref{identification of F} identifies the underlying graded $k$-vector space of $F\Sigma^n k$:
\begin{prop}\label{identification of F}
Let $\Gamma$ be a connective finite-type graded commutative Hopf algebra over a field $k$. Let $n$ be a nonnegative integer, and let $m$ be an integer. Then the $k$-linear dual vector space $\left( (F\Sigma^n k)^m\right)^*$ of the degree $m$ summand $(F\Sigma^n k)^m$ of $F\Sigma^n k$ is isomorphic to the degree $n-m$ summand $\Gamma^{n-m}$ of $\Gamma$. 
That is, $\left( (F\Sigma^n k)^m\right)^* \cong \Gamma^{n-m}$ as $k$-vector spaces.
\end{prop}
\begin{proof}
Using the adjunctions $F\dashv G \dashv E$, we have the isomorphisms of $k$-vector spaces
\begin{align*}
 \hom_{\gr_{\geq 0}\Mod(k)}\left( GF\Sigma^n k, \Sigma^m Gk \right) 
  &\cong \hom_{\gr_{\geq 0}\Comod(k)}\left( F\Sigma^n k, E\Sigma^m Gk \right) \\
  &\cong \hom_{\gr_{\geq 0}\Mod(k)}\left( \Sigma^n k, GE\Sigma^m Gk \right) \\
  &\cong \hom_{\gr_{\geq 0}\Mod(k)}\left( \Sigma^n k, \Sigma^m GEGk \right) \\
  &\cong \hom_{\gr_{\geq 0}\Mod(k)}\left( \Sigma^n k, \Sigma^m G\Gamma \right) \\
  &\cong \Gamma^{n-m}.
\end{align*}
\end{proof}
\begin{corollary}\label{cor on nonstability of F}
Let $\Gamma,n$ be as in Proposition \ref{identification of F}.
Then the projective connective graded $\Gamma$-comodule $F\Sigma^n k$ is trivial in degrees $>n$.

If we furthermore assume that $\Gamma$ is not concentrated in a single grading degree, then there exist positive integers $n$ such that $F\Sigma^n k$ fails to be isomorphic to $\Sigma^n Fk$. That is, the free functor $F: \gr_{\geq 0}\Mod(k)\rightarrow \gr_{\geq 0}\Comod(\Gamma)$ fails to commute with suspension.
\end{corollary}

\begin{example}
Suppose that $\Gamma$ is the mod $2$ dual Steenrod algebra, and $k = \mathbb{F}_2$. Then Proposition \ref{identification of F} gives us that $F\Sigma^0k \cong k$ as a $k$-vector space, and hence also as a $k$-comodule. That is, $F\Sigma^0k\cong \mathbb{F}_2$.
Meanwhile, as a graded $k$-vector space, $F\Sigma^1k$ is isomorphic to $k$ in degree $0$, isomorphic to $k$ in degree $1$, and trivial in all other degrees. So $F\Sigma k$ fails to be isomorphic to $\Sigma Fk$.
\end{example}
\begin{remark}
I hope the reader will forgive me for offering this warning about an easy way to make mistakes when reasoning about the free functor $F$ and the projective connective graded $\Gamma$-comodules $F\Sigma^n k$. 
We adopt the following convenient notation: if $M$ and $N$ are connective graded $k$-vector spaces, we write $\underline{\hom}_{\gr_{\geq 0}\Mod(k)}(M,N)$ for the graded hom-group whose homogeneous degree $n$ summand is trivial for $n<0$, and if $n\geq 0$, it is the set of homomorphisms $M\rightarrow N$ which increase degree by $n$, i.e., $\underline{\hom}_{\mathcal{C}}(M,N)^n = \hom_{\mathcal{C}}(\Sigma^n M, N)$. This is the natural choice of self-enrichment (in the sense of \cite{MR2177301}) of the category of connective graded $k$-vector spaces, so that we have the isomorphism \[\underline{\hom}_{\gr_{\geq 0}\Mod(k)}(M\otimes_k N,Q) = \underline{\hom}_{\gr_{\geq 0}\Mod(k)}\left(M, \underline{\hom}_{\gr_{\geq 0}\Mod(k)}(N,Q)\right)\]
of connective graded $k$-modules for all connective graded $k$-vector spaces $M,N,Q$.

The forgetful functor $G: \gr_{\geq 0}\Comod(\Gamma)\rightarrow \gr_{\geq 0}\Mod(k)$ and the extended comodule functor $E: \gr_{\geq 0}\Mod(k) \rightarrow \gr_{\geq 0}\Comod(\Gamma)$ each commute with suspension. That is, $G\circ \Sigma^n \simeq \Sigma^n \circ G$ and $E\circ \Sigma^n \simeq \Sigma^n \circ E$ for all nonnegative integers $n$. 
It is easy to imagine that the isomorphism of $k$-vector spaces
\begin{align*}
 \hom_{\gr_{\geq 0}\Mod(k)}\left( GFM,N\right)
  &\cong \hom_{\gr_{\geq 0}\Mod(k)}\left( M,GEN\right),
\end{align*}
which we have for all connective graded $k$-vector modules $M$ and $N$, ought to imply the existence of an isomorphism
\begin{align}
\label{iso 230942} \underline{\hom}_{\gr_{\geq 0}\Mod(k)}\left( GFM,N\right)
  &\cong \underline{\hom}_{\gr_{\geq 0}\Mod(k)}\left( M,GEN\right)
\end{align}
of graded $k$-vector spaces. However, there is no such adjunction \eqref{iso 230942}! If we had the isomorphism \eqref{iso 230942} for all connective graded $k$-modules $M$ and $N$, then in the case $N = \Sigma^m k$, we would have the chain of isomorphism of $k$-vector spaces
\begin{align*}
 \hom_{\gr_{\geq 0}\Mod(k)}\left( \Sigma^j GFM,\Sigma^m k\right) 
  &\cong \underline{\hom}_{\gr_{\geq 0}\Mod(k)}\left( GFM,\Sigma^m k\right)^j \\
  &\cong \underline{\hom}_{\gr_{\geq 0}\Mod(k)}\left( M,GE\Sigma^m k\right)^j \\
  &\cong \hom_{\gr_{\geq 0}\Mod(k)}\left( \Sigma^j M,\Sigma^m GEk\right) \\
  &\cong \hom_{\gr_{\geq 0}\Mod(k)}\left( \Sigma^j M,GE\Sigma^m k\right) \\
  &\cong \hom_{\gr_{\geq 0}\Mod(k)}\left( GF\Sigma^j M,\Sigma^m k\right).
\end{align*}
That is, the natural map $GF\Sigma^j M \rightarrow \Sigma^j GFM$ of graded $k$-vector spaces induces an isomorphism on $k$-linear duals in each grading degree. Hence $GF\Sigma^j M \rightarrow \Sigma^j GFM$ is an isomorphism. Since $G$ reflects isomorphisms and commutes with suspension, $F\Sigma^j M \rightarrow \Sigma^j FM$ is an isomorphism, contradicting the failure of $F$ to commute with suspension, demonstrated in Corollary \ref{cor on nonstability of F}. Hence we cannot have an adjunction of the form \eqref{iso 230942}.
\end{remark}

\medskip 

Another corollary of Proposition \ref{identification of F} is an identification of the $k$-vector space underlying the generator $\coprod_n F\Sigma^n k$ for the category of connective graded $\Gamma$-comodules:
\begin{corollary}\label{generator identification 1}
Let $\Gamma,n$ be as in Proposition \ref{identification of F}. Recall that we have the projective generator $F\left( \coprod_n \Sigma^n k\right)$ for the category $\gr_{\geq 0}\Comod(\Gamma)$ of connective graded $\Gamma$-comodules constructed in Theorem \ref{main thm}. Then the degree $m$ summand of $F\left( \coprod_n \Sigma^n k\right)$ has $k$-linear dual vector space isomorphic to the $k$-vector space product $\prod_{n\geq 0} \Gamma^{n-m}$.
\end{corollary}

\subsection{Failure of the category of connective graded comodules to be equivalent to the category of connective graded modules over a ring.}

Corollary \ref{generator identification 1} identified the generator $\coprod_n F\Sigma^n k$ for the category of connective graded $\Gamma$-comodules, but only as a $k$-vector space. In order to prove our main result in this section, Theorem \ref{main nonequivalence thm}, we will need slightly more information about the structure of $\coprod_n F\Sigma^n k$ as a $\Gamma$-comodule.

The needed information will be expressed in terms of the covariant embedding of $\Gamma$-comodules into $\Gamma^*$-modules. This construction is classical: I do not know its historically earliest appearance in the literature, but see \cite{MR686116} for a discussion from a topological perspective, or \cite{MR2012570} for a discussion from a purely algebraic perspective. The construction goes as follows: given a graded $\Gamma$-comodule $M$ with coaction map $\psi: M \rightarrow \Gamma\otimes_k M$, the action map $M\times \Gamma^* \rightarrow M$ sends a pair $(m,f)\in M\times\Gamma^*$ to image of $m$ under the composite
\begin{equation}\label{composite 035494} M\stackrel{\psi}{\longrightarrow} \Gamma\otimes_k M \stackrel{f\otimes M}{\longrightarrow} k\otimes_k M \stackrel{\cong}{\longrightarrow} M. \end{equation}
This action of $\Gamma^*$ on $M$ is called the {\em adjoint action}. 
This construction yields a covariant, exact, faithful, full functor $\Cov: \gr\Comod(\Gamma)\rightarrow \gr\Mod(\Gamma^*)$ which admits a right adjoint $\Rat: \gr\Mod(\Gamma^*)\rightarrow\gr\Comod(\Gamma)$. Given a graded $\Gamma^*$-module, the graded $\Gamma^*$-module $\Cov(\Rat(M))$ is called the {\em rational submodule of $M$,} and it is indeed a graded $\Gamma^*$-submodule of $M$, via the counit map $\Cov(\Rat(M))\rightarrow M$ of the adjunction $\Cov\dashv\Rat$.
See section 4 of \cite{MR2012570} for a presentation of these well-known results. These results rely on $\Gamma$ being projective as a $k$-module, which of course is automatic from our assumption that $k$ is a field. 

To avoid potential confusion, it is important to fix our grading convention for the $k$-linear dual of a graded $k$-vector space. Given a graded $k$-vector space $V$, we grade its $k$-linear dual $V^*$ as follows: the degree $n$ summand of $V^*$ is the $k$-linear dual of the degree $-n$ summand of $V^{-n}$. That is, $(V^*)^n = (V^{-n})^*$. Consequently, if $\Gamma$ is a connective graded Hopf algebra over $k$, then its $k$-linear dual $\Gamma^*$ is a {\em co-connective} graded Hopf algebra over $k$. With this convention, the covariant embedding $\Cov: \gr\Comod(\Gamma)\rightarrow\gr\Mod(\Gamma^*)$ {\em preserves} the gradings.

\begin{prop}\label{identification of F 2}
Let $\Gamma$ be a connective finite-type graded commutative Hopf algebra over a field $k$. Let $n$ be a nonnegative integer. Let $I_n$ denote the two-sided ideal of the co-connective dual Hopf algebra $\Gamma^*$ generated by all homogeneous elements of degree $<-n$. Then, for every connective graded $\Gamma$-comodule $M$, we have an isomorphism of $k$-vector spaces
\begin{equation}
\label{iso 99080} \hom_{\gr_{\geq 0}\Mod(\Gamma^*)}\left( \Cov F\Sigma^n k, \Cov M\right)
  \cong  \hom_{\gr_{\geq 0}\Mod(\Gamma^*)}\left( \Sigma^n \Gamma^*/I_n, \Cov M\right),
\end{equation}
natural in the variable $M$.
\end{prop}
\begin{proof}
For each connective graded $\Gamma$-comodule $M$, we have a chain of adjunction isomorphisms
\begin{align}
\label{eq 99081} \hom_{\gr_{\geq 0}\Mod(\Gamma^*)}\left( \Cov F\Sigma^n k, \Cov M\right) 
          &\cong \hom_{\gr_{\geq 0}\Comod(\Gamma)}\left( F\Sigma^n k, M\right) \\
\nonumber &\cong \hom_{\gr_{\geq 0}\Mod(k)}\left( \Sigma^n k, GM\right) \\
\nonumber &\cong M^n \\
\nonumber &\cong \hom_{\gr\Mod(\Gamma^*)}\left( \Sigma^n \Gamma^*, \Cov M\right) \\
\label{eq 99082}  &\cong \hom_{\gr_{\geq 0}\Mod(\Gamma^*)}\left( \Sigma^n \Gamma^*/I_n, \Cov M\right).
\end{align}
The isomorphism of \eqref{eq 99081} with \eqref{eq 99082} is induced by the natural map 
\begin{equation*}
\Sigma^n\Gamma^*/I_n \rightarrow \Cov F\Sigma^nk\end{equation*} which picks out the copy of $k$ in degree $n$ of $F\Sigma^n k$, so the chain of isomorphisms from \eqref{eq 99081} to \eqref{eq 99082} yields naturality of \eqref{iso 99080} in the variable $M$.
\end{proof}

\begin{theorem}\label{main nonequivalence thm} 
Let $\Gamma$ be a connected\footnote{We emphasize that here we assume {\em connectedness}, not only connectivity. That is, not only is $\Gamma$ trivial in negative degrees: it is also assumed that $\Gamma$ is isomorphic to $k$ in degree zero.} finite-type graded commutative Hopf algebra over a field $k$. 
Suppose that $\Gamma$ is nontrivial in infinitely many degrees.
Then the category $\gr_{\geq 0}\Comod(\Gamma)$ of connective graded $\Gamma$-comodules is {\em not} equivalent to the category of connective graded modules over a graded ring by a suspension-preserving equivalence of categories.
\end{theorem}
\begin{proof}
We argue by contrapositive. The category of connective graded modules over a graded ring does not have a compact projective generator, but what it {\em does} have is a compact projective object $C$ such that the coproduct $\coprod_{n\geq 0}\Sigma^n C$ is a generator. (Namely, $C = R$, as a free $R$-module generated in degree zero.)
So suppose that $\gr_{\geq 0}\Comod(\Gamma)$ has a compact projective object $C$ such that the coproduct $\coprod_{n\geq 0}\Sigma^n C$ is a generator for $\gr_{\geq 0}\Comod(\Gamma)$. In the proof of Theorem \ref{main thm}, we showed that $\coprod_{n\geq 0} F\Sigma^n k$ is a generator for $\gr_{\geq 0}\Comod(\Gamma)$. Consequently there exists an epimorphism $\epsilon: \left( \coprod_{n\geq 0}F\Sigma^n k\right)^{\oplus \kappa} \rightarrow C$ in $\gr_{\geq 0}\Comod(\Gamma)$ for some cardinal number $\kappa$. Since $C$ is projective, the epimorphism $\epsilon$ splits. Choose a section $\sigma: C \rightarrow \left( \coprod_{n\geq 0}F\Sigma^n k\right)^{\oplus \kappa}$ of $\epsilon$ in the category $\gr_{\geq 0}\Comod(\Gamma)$.

We claim that the image of $\sigma$ contains nonzero elements of at most finitely many of the summands $F\Sigma^n k$ of $\left( \coprod_{n\geq 0}F\Sigma^n k\right)^{\oplus \kappa}$. This is easily seen: since $C$ is compact, the natural morphism
\begin{align}\label{nat map 34094309}
 \coprod_{n\geq 0} \coprod_{\kappa} \hom_{\gr_{\geq 0}\Comod(\Gamma)}\left( C,F\Sigma^n k\right) \rightarrow 
 \hom_{\gr_{\geq 0}\Comod(\Gamma)}\left( C,\left( \coprod_{n\geq 0}F\Sigma^n k\right)^{\oplus \kappa}\right)
\end{align}
is an isomorphism, and elements of the direct sum in the domain of \eqref{nat map 34094309} are zero except in finitely many of the summands. 

Since $\im\sigma$ nontrivially intersects only finitely many of the summands $F\Sigma^n k$, there exists some largest integer $N$ such that $\im\sigma$ nontrivially intersects summands of the form $F\Sigma^Nk$. Consequently the integer $N$ and the connective graded $\Gamma$-comodule $C$ have the following properties:
\begin{enumerate}
\item $C$ is a coproduct of retracts of copies of $F\Sigma^n k$ for $n\leq N$.
\item For every connective graded $\Gamma$-comodule $M$ and every homogeneous element $m\in M$, there exists a coproduct $\tilde{C}$ of suspensions of copies of $C$ and a graded $\Gamma$-comodule morphism $\tilde{C}\rightarrow M$ whose image contains $m$.
\end{enumerate}
As a consequence of 1 and 2, we have the following:
\begin{enumerate}
\item[(3)] For every connective graded $\Gamma$-comodule $M$ and every homogeneous element $m\in M$, there exists a coproduct $\tilde{C}$ of suspensions of copies of $F\Sigma^n k$, for various integers $n\leq N$, and a graded $\Gamma$-comodule morphism $\tilde{C}\rightarrow M$ whose image contains $m$.
\end{enumerate}
Bringing Proposition \ref{identification of F 2} to bear now yields:
\begin{enumerate}
\item[(4)] For every connective graded $\Gamma$-comodule $M$ and every homogeneous element $m\in M$, there exists a coproduct $\overline{C}$ of suspensions of copies of $\Sigma^n \Gamma^*/I_n$, for various integers $n\leq N$, and a graded $\Gamma^*$-module morphism $\overline{C}\rightarrow \Cov(M)$ whose image contains $m$,
\end{enumerate}
and consequently,
\begin{enumerate}
\item[(5)] For every connective graded $\Gamma$-comodule $M$ and every homogeneous element $m\in M$, the element $m\in M$ is annihilated by the ideal $I_N$ of $\Gamma^*$
\end{enumerate}

Recall our assumption that $\Gamma$ is nontrivial in infinitely many degrees. 
Consequently there must be some nonzero homogeneous element of $I_0$ which is in a degree $d$ with $d<-N$. Choose such an element in $\Gamma^*$ in degree $d$, and call it $\gamma^*$. Choose also a homogeneous element $\gamma\in \Gamma$ in degree $-d$ such that $\gamma^*$, when evaluated on $\gamma$, is equal to $1\in k$.
By counitality and connectedness of $\Gamma$, we have
\[ \Delta(\gamma) = \gamma\otimes 1 + 1\otimes \gamma\mod J_0\otimes J_0,\]
where $J_0$ is the augmentation ideal in $\Gamma$.
(To avoid possible confusion, we remind the reader that 
the notation $I_0$ already is reserved for the augmentation ideal in $\Gamma^*$.) 
Then, following the recipe for the adjoint action of $\Gamma^*$ on $\Gamma$ described above in \eqref{composite 035494}, we have that $\gamma^*\cdot \gamma = 1$. 
Consequently we have an element $\gamma$ in the graded $\Gamma^*$-module $\Cov(\Gamma)$ such that $\gamma$ is not $\gamma^*$-torsion. But since $\gamma^*$ is in degree $d<-N$, $\gamma^*$ is in the ideal $I_N$ of $\Gamma^*$ generated by all elements of grading degree $<-N$. Hence $\Cov(\Gamma)$ contains a homogeneous element which is not $I_N$-torsion. This contradicts claim (5) above, whose truth we already established. So $C$ must not exist.
\end{proof}

\begin{corollary}\label{steenrod nonequivalence cor} 
Let $p$ be a prime number. Then the category of connective graded comodules over the mod $p$ dual Steenrod algebra is not equivalent, via a suspension-preserving functor, to the category of connective graded modules over any ring.
\end{corollary}

\bibliography{/home/asalch/texmf/tex/salch}{}
\bibliographystyle{plain}
\end{document}